\documentclass[11pt]{article}
\usepackage{amsmath,amssymb,amsfonts}
\usepackage{graphicx,psfrag,epsf}
\usepackage{enumerate}
\usepackage{url} % not crucial - just used below for the URL 

\usepackage{subcaption}
\usepackage{float}

\newtheorem{Theorem}{Theorem}[section]
\newtheorem{Proposition}[Theorem]{Proposition}
\newtheorem{Definition}[Theorem]{Definition}
\newtheorem{Lemma}[Theorem]{Lemma}
\newtheorem{Corollary}[Theorem]{Corollary}
\newtheorem{Example}[Theorem]{Example}

\newtheorem{Remark}[Theorem]{Remark}
\newenvironment{proof}{\emph{Proof.}} {\quad \hfill $\blacksquare$ \newline}

\usepackage{color}

\date{}
  \title{\bf On partial stochastic comparisons based on tail values at risk}
  \author{Alfonso J. Bello\\
   Dpto.\ Estad\'istica e Investigaci\'on Operativa, Universidad de C\'adiz \\
Facultad de Ciencias, Campus Universitario R\'io San Pedro s/n \\
11510 Puerto Real (C\'adiz), SPAIN \\ 
\texttt{alfonsojose.bello@uca.es}\\
    and\\
    Julio Mulero\\
   Dpto. Matem\'aticas, Universidad de Alicante \\
Facultad de Ciencias, Apartado de correos 99 \\
03080 Alicante, SPAIN \\
\texttt{julio.mulero@ua.es}\\
    and \\
    Miguel A. Sordo\\
    Dpto.\ Estad\'istica e Investigaci\'on Operativa, Universidad de C\'adiz \\
Facultad de Ciencias, Campus Universitario R\'io San Pedro s/n \\
11510 Puerto Real (C\'adiz), SPAIN \\ 
\texttt{mangel.sordo@uca.es}\\
and \\
    Alfonso Su\'arez-Llorens\\
    Dpto.\ Estad\'istica e Investigaci\'on Operativa, Universidad de C\'adiz \\
Facultad de Ciencias, Campus Universitario R\'io San Pedro s/n \\
11510 Puerto Real (C\'adiz), SPAIN \\ 
\texttt{alfonso.suarez@uca.es}}

\begin{document}

\def\spacingset#1{\renewcommand{\baselinestretch}%
{#1}\small\normalsize} \spacingset{1}

%%%%%%%%%%%%%%%%%%%%%%%%%%%%%%%%%%%%%%%%%%%%%%%%%%%%%%%%%%%%%%%%%%%%%%%%%%%%%%

\maketitle

\begin{abstract}
In risk theory, financial asset returns often follow heavy-tailed distributions. Investors and risk managers used to compare risk measures as the value at risk or tail value at risk in order over the whole confidence levels to avoid the exposure to to large risks. In this paper we analyze the comparison between tail values at risk from a confidence level and beyond which is a reasonable criterion when we are focused on large losses or simply we cannot give a complete ordering over all the confidence levels. A family of stochastic orders indexed by $p_0\in(0,1)$ is proposed. We study
their properties and connections with other classical criteria as
the increasing convex and tail convex orders and we rank some parametrical families of distributions.  Finally, two applications with real datasets are given as well.
\end{abstract}

\noindent%

\section{Motivation and Preliminaries}\label{sec:1}
In actuarial and financial sciences, risk managers and investors dealing with insurance losses and asset returns are often concerned with the right-tail risk of distributions, which is related to large deviations due to the right-tail losses or right-tail returns (see Wang [1]).  To compare the risk associated to different models, they use risk measures typically based on quantiles, such as the value at risk, the tail value at risk and some generalizations of these measures (see the book by Guégan and Hassani [2] for a recent review on the topic of risk measurement). Comparisons of two quantile-based risk measures made for a particular confidence level $p$ are not very informative (they are based on two single numbers) and a change in the confidence level may produce different conclusions. An~alternative method to compare two risks $X$ and $Y$ is to use stochastic orderings, which conclude that $X$ is smaller than $Y$ when a family of risk measures (rather than a single risk measure)  agrees in the conclusion that $X$ is less risky than $Y.$ %Advantages and disadvantages of both methods are those derived from making cardinal and ordinal comparisons.
For example, the increasing convex order (which is formally defined below) requires agreement of tail values at risk for any confidence level $p \in (0,1)$. An advantage of this method is, obviously, its robustness toward changes in the confidence level, which~can be interpreted in terms of a common agreement of different decision-maker's attitudes. %Please confirm this was the intended meaning
By robustness, we understand how sensitive the comparison procedure is to different values of $p$. For other interpretations of robustness in risk measurement, see~Zhelonkin and Chavez-Demoulin [3].    However, this approach has the drawback that some distributions cannot be compared despite our intuition that one is less risky than the other. 

One way to increase the number of distributions that can be compared without losing robustness is to reduce the range of values of the confidence level $p$ required to reach an agreement. For example, an investor concerned with right-tail risks may think that $X$ is less risky than $Y$ if the respective tail values at risk agree in the conclusion that $X$ is less risky than $Y$ for any confidence level $p$ such that $p > p_0,$ where $p_0 \in (0,1)$ is chosen to suit his/her specific preferences. The aim of this paper is precisely to study a family of stochastic orders based on such agreement.

Let $X$ be a random variable describing losses of a financial asset and let $F$ be its distribution function (as usual, negative losses are gains). The value at risk of $X$ at level $p\in(0, 1),$ or $p$-quantile, is defined by 
\begin{equation}\label{quantile}
\textup{VaR}_p[X]= F^{-1}(p)=\inf\{x: F(x)\geq p\}, \text{ for all } p\in(0,1).
\end{equation}

For a given $p\in(0,1)$, $\textup{VaR}_p[X]$ represents the maximum loss  the investor can 
suffer with $100p\%$ confidence over a certain period of time.  Despite being widely used in practice, VaR has two major drawbacks. First, it does not describe the tail behavior beyond the confidence level. Second, VaR is not, in general, subadditive (except in some special cases; for example, when the underlying distribution belongs to the elliptical family of distributions, the estimator of  VaR is subadditive). The~literature offers different possible alternatives to overcome the limitations of VaR (see, for example, Ahmadi-Javid [4]). One of the most important is the tail value at risk or $\textup{TVaR}$, defined by 

\begin{equation*}
\textup{TVaR}_p[X] =\dfrac{1}{1-p}\int_p^1 F^{-1}(u)du, \text{ for all } p\in[0,1).
\end{equation*}

TVaR is subadditive (that is, $\textup{TVaR}_p[X+Y]\le \textup{TVaR}_p[X]+\textup{TVaR}_p[Y]$ for all losses $X$ and $Y$). Given~$p\in(0,1)$ and $X$ continuous, $\textup{TVaR}_p[X]$ represents the average  loss when losses exceed~$\textup{TVaR}_p[X]$. 

 In this paper, we study the following family of stochastic orders indexed by confidence levels $p_0 \in [0,1).$
\begin{Definition} \label{p0tvar} 
Let $X$ and $Y$ be two random variables and $p_0\in[0,1)$. Then, $X$ is said to be smaller than $Y$ in the $p_0$-tail value at risk order, denoted by $X\leq_{p_0\textup{-tvar}} Y$, if 
\begin{equation}\label{tvp}
\textup{TVaR}_{p}(X)\leq\textup{TVaR}_{p}(Y),\text{ for all }p> p_{0}.
\end{equation}
\end{Definition}
When (\ref{tvp}) holds for $p_0=0$, then we have the usual increasing convex order (see Lemma 2.1 in Sordo and Ramos [5]) and we denote $X\le_{icx}Y$. The books by Shaked and  Shanthikumar [6] and Belzunce et al. [7] collect many properties and applications of the increasing convex order. If $X\leq_{\textup{icx}}Y$, then $X$ is both smaller (in fact, $X\le_{icx}Y$ implies $\textup{E}[X] \leq \textup{E}[Y]$) and less variable than $Y.$ In finance and insurance, the increasing convex order is often interpreted in terms of stop-loss contracts and it is also called stop-loss order. Specifically, it holds that $X\leq_{\textup{icx}} Y$ if, and only if,
 \begin{equation}\label{charac_icx1}
\textup{E}[(X-x)_+]\leq \textup{E}[(Y-x)_+], \text{ for all }x\in\mathbb R,
\end{equation}
where $(x)_+=x$, if $x\ge 0$ and $(x)_+=0$, if $x<0$, or, equivalently, if 
\begin{equation}\label{charac_icx2}
\int_{x}^{+\infty}\overline F(t)dt\leq \int_{x}^{+\infty}\overline G(t)dt, 
\text{ for all }x\in\mathbb R,
\end{equation}
provided the integrals exist, where $\overline F=1-F$ and $\overline{G}=1-G$ are  the tail functions of $X$ and $Y$, respectively.  One purpose of this paper is to study how these characterizations change when the order $X\leq_{\textup{icx}} Y$ is replaced by the weaker    $X\leq_{p_0\textup{-tvar}} Y,$ for some $p_0>0$.

We provide an analytical example to motivate the study of the new family of stochastic orders. Let us consider two Pareto loss distributions $X \sim P(7,3)$ and $Y \sim P(3,2)$. The value at risk of a Pareto distribution $Z\sim P(a,k)$, with 
shape parameter $a>0$ and scale parameter $k \in\mathbb{R}$, is given by
\begin{equation*}
	F_Z^{-1}(p)=\frac{k}{(1-p)^{\frac{1}{a}}}, \text{ for all } p \in (0,1),
\end{equation*}
and, if $a>1,$ the expectation is $\textup{E}[Z]=a k/(a-1)$.  Figure \ref{motivation1} 
shows  the plot of the quantile functions for $X$ and $Y$.   Since $\textup{E}[X] = 
3.5>3=\textup{E}[Y]$, then  $X \nleq_{icx} Y$.  
\begin{figure}[H]
	\centering
	\includegraphics[width=0.5\textwidth]{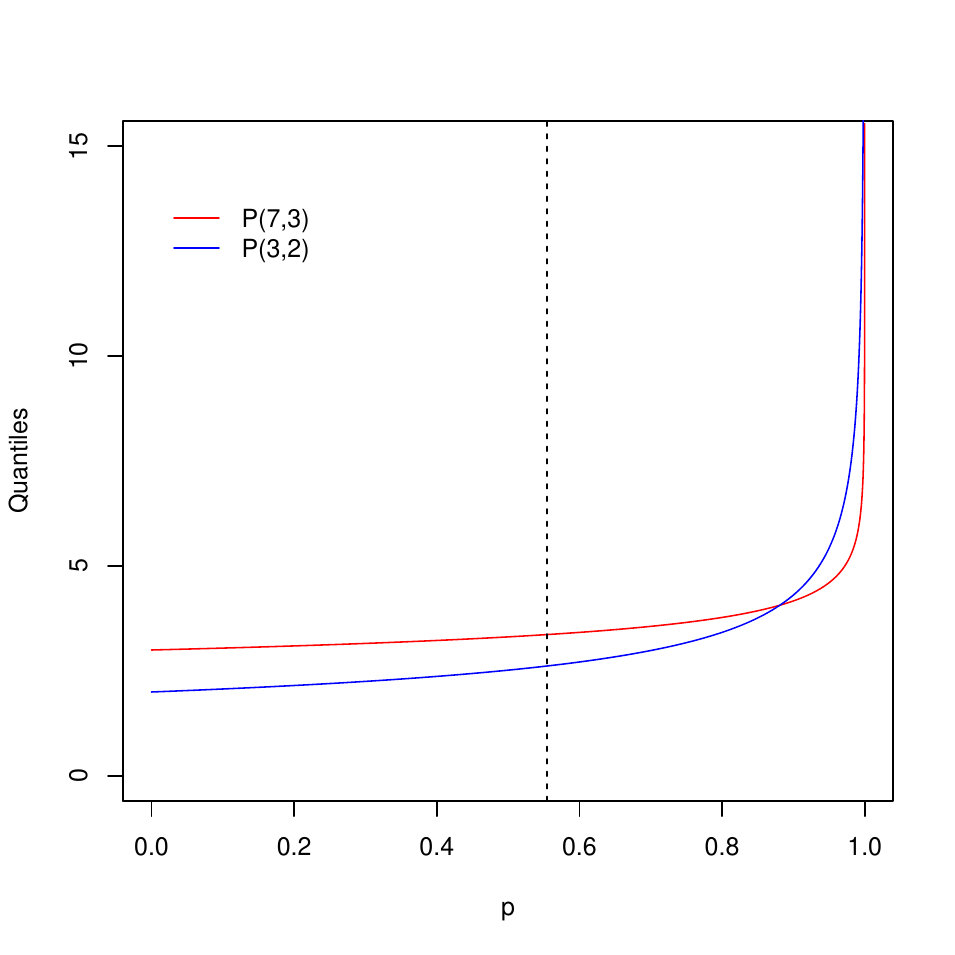}%{remark52HR.pdf}
	\caption{Quantile functions for  $X \sim P(7,3)$ and $Y \sim P(3,2)$.}\label{motivation1}
\end{figure}
However, it can be checked that  $\textup{TVaR}_{p}(X)\leq\textup{TVaR}_{p}(Y)$, for all $p\geq 0.55482$ where here \mbox{$p_0=0.55482$} is the minimum value such that $X\leq_{p_0\textup{-tvar}} Y$ ($p_0$ is represented with a dashed line in Figure \ref{motivation1}). An investor concerned with large deviations due to the right-tail losses may make decisions based on the tail value at risk for large values of $p$. If this is the case, he/she will evaluate $X$ as less dangerous than $Y,$ despite $\textup{E}[X]>\textup{E}[Y]$ and $X \nleq_{\textup{icx}} Y$.

The idea of limiting the number of comparisons to weaken the increasing convex order is not new. Given two random variables $X$ and $Y$ with the same mean,  Cheung and Vanduffel [8] say that $X$ is smaller than $Y$ in the tail convex order with index $x_0 \in \mathbb{R}$ (denoted by $X \leq_{\textup{tcx}(x_0)} Y$) if 
\begin{equation}\label{tcx_def}
\textup{E}[(X-x)_+]\leq \textup{E}[(Y-x)_+], \text{ for all }x\geq x_0.
\end{equation}

According to (\ref{charac_icx1}), if $X$ and $Y$ have the same mean, then $X \leq_{\textup{tcx}(-\infty)} Y$ is the same as $X\le_{icx}Y.$ It is natural to wonder, when $E[X]=E[Y]$, about the relationship between $X\le_{p_0\textup{-tvar}}Y$ and  $X \leq_{\textup{tcx}(x_0)} Y.$ We address this issue below.

For other recent studies of orders based on tail comparisons see
Sordo et al. [9], Mulero et al. [10] and Belzunce et al. [11].

The rest of the work is organized as follows. In Section~\ref{sec:2}, we study the main properties of the order $\le_{p_0\text{-tvar}}$ and find sufficient conditions under which it holds. We also investigate its relationships with other well-known stochastic orders and compare parametric families of distributions. In particular, we relate the order  $\leq_{p_0-tvar}$  to the notion of pure tail order as considered by Rojo [12]. In~Section~\ref{sec:5} we illustrate the applicability of the new family of orders using a real dataset. Section \ref{sec:6} contains~conclusions. 

Throughout the paper, ``increasing'' means ``non-decreasing'' and ``decreasing'' means ``non-increasing''. Random variables are assumed to have finite means. Given a function $h$, $S^-(h)$ denotes the number of sign changes of $h$ on its support, where zero terms are discarded.

\section{Properties and Relationships  with Other Stochastic Orders}\label{sec:2}

It is trivially verified that $X\leq_{p_0\textup{-tvar}} X$ for all $p_0\in(0,1)$ and that $X \leq_{p_0\textup{-tvar}}Y$ for $p_0\in(0,1)$ implies $X \leq_{q_0\textup{-tvar}}Y$, for all $q_0\in[p_0,1)$. 
It is also obvious that if $X$, $Y$ and $Z$ are three random variables such that $X\leq_{p_0\textup{-tvar}} Y$ and $Y\leq_{p_1\textup{-tvar}} Z$ with $p_0, p_1\in(0,1)$, then $X\leq_{p_3\textup{-tvar}} Z$ with $p_3=\max\{p_0,p_1\}$. Finally, if $X\leq_{p_0\textup{-tvar}} Y$ and $Y\leq_{p_1\textup{-tvar}} X$,  with $p_0,p_1\in(0,1)$, then $\textup{TVaR}_{p}(X)=\textup{TVaR}_{p}(Y)$, for all $p\in[p_3,1]$, where $p_3$ is defined as before.

Next, we provide some closure properties. In particular, we point out the closure under convergence in distribution and under increasing convex transformations. First of all, recall that  given a sequence of random variables $\{X_n:n=1,2,\dots\}$ with distribution functions $F_n$ and $F$, respectively, then $X_n$ is said to converge in distribution to $X$, denoted by $X_n\overset{\textit{d}}{\longrightarrow} X$, if $\lim_{n\rightarrow +\infty}F_n(x)=F(x)$ for all $x$ at which $F$ is continuous.

\begin{Proposition}
Let $\{X_n:n=1,2,\dots\}$ and $\{Y_n:n=1,2,\dots\}$ be two sequences of positive continuous random variables such that $X_n\overset{\text{d}}{\longrightarrow} X$ and $Y_n\overset{\text{d}}{\longrightarrow} Y$. Assume that $X_n$ and $X$ have a common interval support for all $n\in\mathbb N$ and $\lim_{n \rightarrow \infty}\textup{E}[X_n]=\textup{E}[X]$ (and, analogously, for $Y_n$ and $Y$ and their expectations). If~$X_n\leq_{p_0\textup{-tvar}}Y_n$ for $p_0\in(0,1)$, then $X\leq_{p_0\textup{-tvar}}Y$.
\end{Proposition}
\begin{proof}From Definition \ref{p0tvar}, it is clear  that $X\leq_{p_0\textup{-tvar}}Y$ if, and only if,
		\begin{equation}\label{charac_p0tvar1}
	\int_{p}^{1}F^{-1}(u)du\leq \int_{p}^{1}G^{-1}(u)du, \text{ for all } p\geq p_0.
	\end{equation}
	
	Integration by parts in \eqref{charac_p0tvar1} shows that
		\begin{equation}\label{charac_p0tvar2}
	X\leq_{p_0\textup{-tvar}}Y \Longleftrightarrow \int_{F^{-1}(p)}^{+\infty}\overline F(t)dt + F^{-1}(p)(1-p)\leq \int_{G^{-1}(p)}^{+\infty}\overline G(t)dt 
	+ G^{-1}(p)(1-p), p\geq p_{0}.
	\end{equation}
	
Now, from $X_n\overset{\text{d}}{\longrightarrow} X$, we have that
	\begin{equation}
	\lim_{n\rightarrow+\infty}F^{-1}_n(p)=  F^{-1}(p),\label{convergence1}
	\end{equation}
for all point $p$ where $F^{-1}$ is continuous, where $F_n$ and $F$ are the distribution functions of $X_n$ and $X$, respectively (see, for example, Chapter 21 in Van der Vaart [13]).  Moreover, from Theorem 2.3 in M\"uller~[14], it holds that
	\begin{equation}\label{convergence2}
	\lim_{n\rightarrow+\infty}\int_{x}^{+\infty}\overline F_n(t)dt=\int_{x}^{+\infty}\overline F(t)dt,\text{ for all }x\in\mathbb R.
	\end{equation}
	
	Using the dominated convergence theorem, it follows from (\ref{convergence1}) and (\ref{convergence2}) that
	 \begin{equation} \label{af}
	 \lim_{n\rightarrow+\infty}\left[\int_{F^{-1}_n(p)}^{+\infty}\overline F_n(t)dt+ F^{-1}_n(p)(1-p)\right]= \int_{F^{-1}(p)}^{+\infty}\overline F(t)dt+F^{-1}(p)(1-p).
	 \end{equation}
	 
From $Y_n\overset{\text{d}}{\longrightarrow} Y$ we see that (\ref{af}) also holds when $G$ and $G_n$ replace, respectively, $F$ and $F_n$. Now~the result follows from (\ref{charac_p0tvar2}).
\end{proof}

The closure under strictly increasing convex transformations require the following lemma taken from page 120 in Barlow and Proschan [15].
\begin{Lemma}\label{lemmaBP}
Let $W$ be a measure on the interval $(a,b)$, not necessarily nonnegative. Let $h$ be a nonnegative function defined on $(a,b)$.  If $\int_t^b dW(x)\geq 0$ for all $t\in (a,b)$ and if $h$ is increasing, then $\int_a^b h(x)dW(x)\geq 0$.
\end{Lemma}

\begin{Proposition}\label{closuretrans}
Let $X$ and $Y$ be two random variables and let $\phi$ be a strictly increasing and convex function. If~$X\leq_{p_0\textup{-tvar}}Y$ for $p_0\in (0,1)$, then $\phi(X)\leq_{p_0\textup{-tvar}}\phi(Y)$.
\end{Proposition}
\begin{proof}
Let $F_\phi$ and $G_\phi$ denote the distribution functions of $\phi(X)$ and $\phi(Y)$, respectively. Since $\phi$ is strictly increasing we see that $F_{\phi}^{-1}(p)=\phi(F^{-1}(p))$ and $G_{\phi}^{-1}(p)=\phi(G^{-1}(p))$, for all $p\in(0,1)$. Since~$\phi$ is convex, we know (see, for example, Theorem 10.11 in Zygmund [16]) that there exists an increasing function $\varphi$ such that 
\begin{equation*}
\phi (G^{-1}(u))-\phi (F^{-1}(u))=\int_{F^{-1}(u)}^{G^{-1}(u)}\varphi (v)dv.  \label{barlowproschanclosure}
\end{equation*}

Moreover, since $\phi$ is strictly increasing, the function $\varphi $ is positive (in fact, it is the derivative of $\phi$ except, perhaps, at some singular points).  
Observe that if $F^{-1}(u)\leq G^{-1}(u)$, then 
		\begin{equation*}
		\int_{F^{-1}(u)}^{G^{-1}(u)}\varphi (v)dv\geq \varphi (F^{-1}(u))\left[
		G^{-1}(u)-F^{-1}(u)\right], 
		\end{equation*}%
		and, if $F^{-1}(u)>G^{-1}(u)$, then 
		\begin{equation*}
		\int_{F^{-1}(u)}^{G^{-1}(u)}\varphi
		(v)dv=-\int_{G^{-1}(u)}^{F^{-1}(u)}\varphi (v)dv\geq \varphi (F^{-1}(u))
		\left[ G^{-1}(u)-F^{-1}(u)\right] .
		\end{equation*}%
		
	Therefore, we have that 
	\begin{eqnarray*}
\int_{p}^{1}\left[ \phi (G^{-1}(u))-\phi (F^{-1}(u))\right] du 
& = &\int_{p}^{1}\left( \int_{F^{-1}(u)}^{G^{-1}(u)}\varphi (v)dv\right) du \\
& \geq & \int_{p}^{1}\varphi (F^{-1}(u))\left[ G^{-1}(u)-F^{-1}(u)\right] du 
\geq 0,
	\end{eqnarray*}
where the last inequality follows from (\ref{charac_p0tvar1})
		and Lemma \ref{lemmaBP} using the fact that the function $\varphi (F^{-1}(u))I(p<u)$ 
is positive and increasing. Consequently,
		\begin{equation*}
		\int_{p}^{1} \phi (F^{-1}(u)) \leq \int_{p}^{1} \phi (F^{-1}(u)) du ,
	\end{equation*}
for all $p\geq p_0$,	which means $\phi(X)\leq_{p_0\textup{-tvar}}\phi(Y)$.
	\end{proof}

Next, we give the following relationship among the new order and the icx order of certain random~variables. 

\begin{Proposition}\label{icxsufficienteconditions}
Let $X$ and $Y$ be two random variables. If 

\[
\max\left\{X,F^{-1}(p_0)\right\}\leq_{\textup{icx}} \max\left\{Y,G^{-1}(p_0)\right\},
\]
then $X\leq_{p_0\textup{-tvar}}Y$.
\end{Proposition}

\begin{proof}
Given a random variable $X$ with distribution function $F$, let 
\begin{equation*}
X_{0}=\max\left\{X,F^{-1}(p_0)\right\}, \ p_0\in(0,1),
\end{equation*}
and let $F^{-1}_{0}$ be its corresponding quantile function, given by
\begin{equation}\label{quantileXp}
F^{-1}_{0}(p)=\left\{\begin{array}{ll}
F^{-1}(p_0), &\text{ if }0<p<p_0,\\
F^{-1}(p), &\text{ if }p_0\leq p<1.
\end{array}\right.
\end{equation}

A straightforward computation gives
%Now, let $X$ and $Y$ be two random variables with quantile functions $F^{-1}$ and $G^{-1}$, respectively. From \eqref{icx_def}, we have that $X_{p_0}\leq_{icx}Y_{p_0}$, for $p_0\in (0,1)$, if, and only if, $\textup{TVaR}_{p}(X_{p_0})\leq \textup{TVaR}_{p}(Y_{p_0})$, for all $p\in(0,1)$. From a straightforward computation we have that the tail value at risk of $X_{p_0}$ is given by
\vspace{6pt}
\begin{equation} \label{mct}
\text{TVaR}_p(X_{0})=\left\{\begin{array}{ll}
\frac{1}{1-p}\left(\int_{p_0}^1F^{-1}(u)du+F^{-1}(p_0)(p_0-p)\right), &\text{ if }0<p<p_0,\\
\frac{1}{1-p}\int_p^1 F^{-1}(u)du=\text{TVaR}_p(X), &\text{ if }p_0\leq p<1.
\end{array}\right.
\end{equation}

We obtain similarly $\text{TVaR}_p(Y_{0}),$ where $Y_0$ is analogously defined  to $X_0$. It is clear than $\text{TVaR}_p(X_{0})\le \text{TVaR}_p(Y_{0}) $ for all $p\in (0,1)$ implies $\text{TVaR}_p(X)\le \text{TVaR}_p(Y) $ for
$p_0\leq p<1$, which~proves the result. 
\end{proof}

\begin{Remark} 
A stop-loss order is defined as the maximum loss that an investor assumes  on a particular investment. Given a loss random variable $X$, the random variable $\max\left\{X,F^{-1}(p_0)\right\}$ can be interpreted as a stop-loss order at $F^{-1}(p_0).$ 
\end{Remark} 

Next we discuss the connections between the  $p_0$-tail value at risk order and the tail convex order defined by \eqref{tcx_def}. Observe that $X\leq_{tcx(x_0)} Y$ if $E[X]=E[Y]$ and
\begin{equation} \label{casitcx}
\int_{x}^{+\infty}\overline F(t)dt\leq \int_{x}^{+\infty}\overline G(t)dt, \ \forall x\geq x_0.
\end{equation}

With our notation, given $X$ and $Y$ such that $E[X]=E[Y],$ Theorem $4$ of Cheung and Lo [17] shows  that $X\leq_{p_0\textup{-tvar}}Y$ for 
$p_0 \in (0,1)$ implies $X\leq _{\textup{tcx}(F^{-1}(p_{0}))}Y$ and that $X\leq _{\textup{tcx}(x_0)}Y$ for $x_0\in\mathbb R$ implies $X\leq _{G(x_{0})\textup{-tvar}}Y$ and $X\leq _{\textup{tcx}(F^{-1}(G(x_{0}))}Y$.

\begin{Remark}
Given $X$ and $Y$ with the same mean, a natural question is whether   $X\leq_{p_0\textup{-tvar}}Y$ for $p_0 \in (0,1)$ implies $X\leq _{\textup{tcx}(x_0)}Y$ for $x_0<F^{-1}(p_{0})$. In general, the answer is no. Let us assume that $X\leq _{p_0\textup{-tvar}}Y$ for $p_0\in(0,1)$ such that $G^{-1}(p_{0})<F^{-1}(p_{0})$ and 
\begin{equation}
\label{eqAreas} \int_{p_{0}}^{1} F^{-1}(u) du = \int_{p_{0}}^{1} G^{-1}(u) du.
\end{equation}

Under these assumptions, $p_0=\{p\in (0,1): X\leq_{p\textup{-tvar} }Y\}\neq \emptyset$. Let us consider $x_0\in(G^{-1}(p_{0}),F^{-1}(p_{0}))$ such that 
\begin{equation}
\label{condAreapeor} \int_{F(x_{0})}^{p_{0}} F^{-1}(u) du - x_{0} \left( p_{0} - F(x_{0}) \right) > 
x_{0} \left( G(x_{0})-p_{0} \right) - \int_{F(x_{0})}^{p_{0}} G^{-1}(u) du.
\end{equation}

Then we have 
\begin{eqnarray}
\int_{x}^{+\infty} \left( \overline{F}(t)-\overline{G}(t) \right) dt &=& 
\int_{F(x_{0})}^{1} F^{-1}(u) du - x_{0} \left( G(x_{0})-F(x_{0}) \right) - 
\int_{G(x_{0})}^{1} G^{-1}(u) du\nonumber\\
&=& \int_{p_{0}}^{1} F^{-1}(u) du - \int_{p_{0}}^{1} G^{-1}(u) du \nonumber\\
&&
 + 
\int_{F(x_{0})}^{p_{0}} F^{-1}(u) du - x_{0} \left( p_{0} - F(x_{0}) \right)\nonumber\\&&+ \int_{p_{0}}^{G(x_{0})} G^{-1}(u) du - x_{0} \left( G(x_{0})-p_{0} \right)\label{eqintfin}.
\end{eqnarray}

From (\ref{eqAreas}), (\ref{condAreapeor}) and (\ref{eqintfin}), it holds that 
\begin{equation*}
\int_{x_{0}}^{+\infty} \left( \overline{F}(t)-\overline{G}(t) \right) dt > 0,
\end{equation*}
and $X\not\leq _{\textup{tcx}(x_0)}Y$. A similar reasoning shows that, in general, $X\leq _{tcx(x_0)}Y$ 
does not imply $X\leq _{p_{0}\textup{-tvar}}Y$ for $p_0<F(x_{0})$.
Figure \ref{DEF_Weibull_Pareto_Quantile} illustrates this situation for $X\sim W(3,1)$ and $Y\sim P(3/2,1)$, where $X$ is Weibull and $Y$ Pareto  with distribution functions $F$ and $G$, respectively. Recall that if $X\sim W(\alpha,\beta)$ then 
\begin{equation*}
F^{-1}(p)=\alpha\left(-\log(1-p)\right)^{1/\beta},
\end{equation*}
for $p\in (0,1)$.  Note that $\textup{E}[X]=\textup{E}[Y]$. It can be checked that  $X\leq _{p_{0}\textup{-tvar}}Y$ with  $p_0=0.68147$ (vertical line in Figure \ref{DEF_Weibull_Pareto_Quantile}) and  $X\leq _{\textup{tcx}(F^{-1}(p_0))}Y$ with $F^{-1}(p_0)=3.43711$ (horizontal line). However, since $G^{-1}(p_{0})<F^{-1}(p_{0})$ and \eqref{eqAreas} holds, $X\not\leq _{\textup{tcx}(x_0)}Y$ for any $x_0<F^{-1}(p_0)$.
\begin{figure} [H]
\centering
\includegraphics[width=0.5\linewidth]{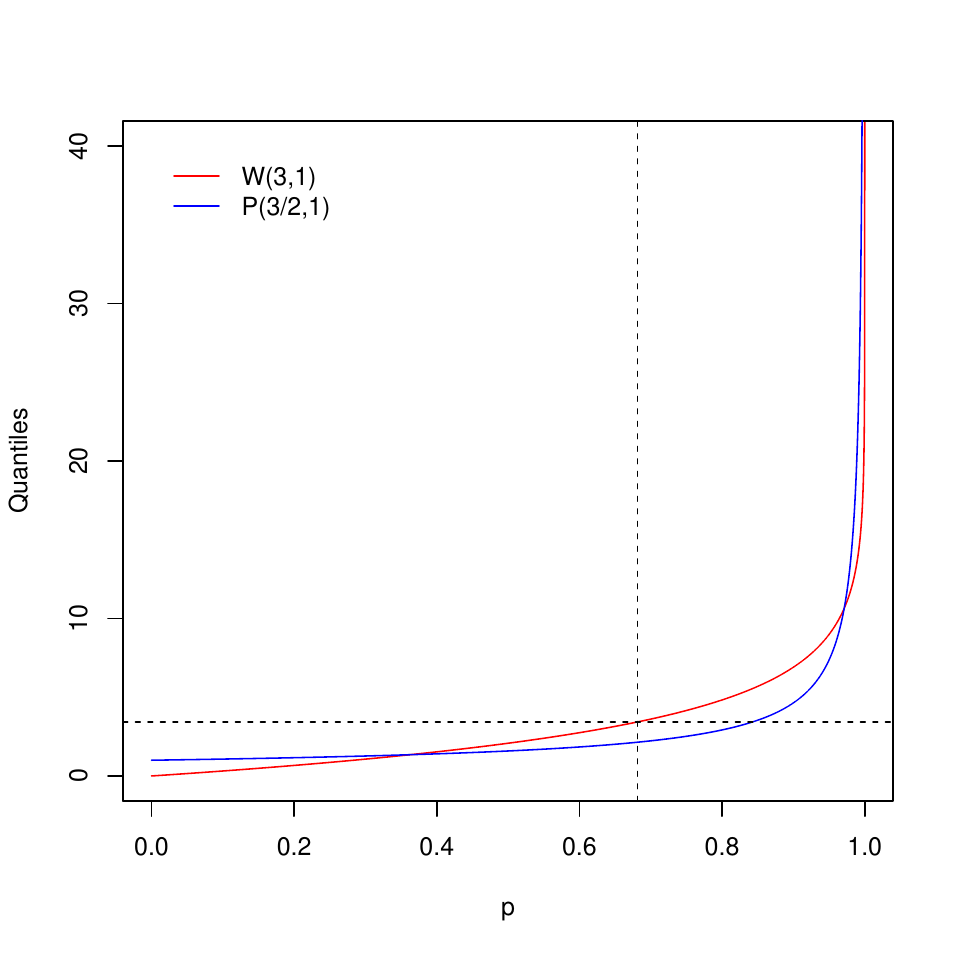}
\caption{Quantile functions for $X\sim W(3,1)$ and $Y\sim P(3/2,1)$.}
\label{DEF_Weibull_Pareto_Quantile} 
\end{figure} 
\end{Remark}

Now consider two general random variables (not necessarily with the same mean). The following result connects the  $p_0$-tail value at risk order and the stochastic order defined by \eqref{casitcx}. 

\begin{Proposition}\label{thm_tvar<-&->tcx}
Let $X$ and $Y$ be two random variables with distribution functions $F$ and $G$, respectively.
\begin{enumerate}
\item[(i)] If $X\leq_{p_0\textup{-tvar}}Y$, then $\int_{x}^{+\infty} \overline{F} (t)  dt \leq 
\int_{x}^{+\infty} \overline{G} (t)  dt$, for all $x \geq F^{-1}(p_{0})$.
\item[(ii)] If $\int_{x}^{+\infty} \overline{F} (t) \ dt \leq 
\int_{x}^{+\infty} \overline{G} (t) \ dt$, for all $x \geq x_{0}$, then 
$X\leq _{G(x_{0})-tvar}Y$.
\end{enumerate}
\end{Proposition}
\begin{proof}
We prove (i) (the proof of (ii) is similar). Assume $X\leq_{p_0\textup{-tvar}}Y$. From (\ref{charac_p0tvar2}) we know that
\begin{equation}
\label{desintdescomp} (1-p)F^{-1}(p)+\int_{F^{-1}(p)}^{+\infty} \overline{F} (t) dt \leq 
(1-p)G^{-1}(p)+\int_{G^{-1}(p)}^{+\infty} \overline{G} (t) dt, \text{ for all } p \geq p_0.
\end{equation}

Fixed $p \geq p_0$, the function
\begin{equation*}
g_p(x)=(1-p)x+\int_{x}^{+\infty} \overline{G} (t) dt,\text{ for all }x\in\mathbb R,
\end{equation*}
reachs its minimum at  $G^{-1}(p)$ (see Theorem 3.2 in Dhaene et al. [18]). 
From this fact and (\ref{desintdescomp}) it follows~that 
\begin{equation*}
(1-p)F^{-1}(p)+\int_{F^{-1}(p)}^{+\infty} \overline{F} (t) dt \leq 
(1-p)F^{-1}(p)+\int_{F^{-1}(p)}^{+\infty} \overline{G} (t) dt, \text{ for all }  p \geq p_0,
\end{equation*} 
or, equivalently, 
\begin{equation}
\label{dessurvf} \int_{x}^{+\infty} \overline{F} (t)dt \leq \int_{x}^{+\infty} \overline{G} (t)dt, 
\text{ for all }  x \geq F^{-1}(x_{0}),
\end{equation}
which concludes the proof.
\end{proof}

\begin{Corollary}
\label{Cor_equiv_tvar_SurvFunct} Let $X$ and $Y$ be two random variables with distribution functions 
$F$ and $G$, respectively, and $p_{0}\in \left( 0,1\right)$ such that $F^{-1}(p_{0}) \leq 
G^{-1}(p_{0})$. Then, $X\leq _{p_{0}-tvar}Y$ if, and only if, one of the following equivalent 
conditions holds:
\begin{itemize}
\item[(i)] $E\left[ \left( X-x \right)_{+} \right] \leq 
E\left[ \left( Y-x \right)_{+} \right]$, for all $x \geq F^{-1}(p_{0})$.
\item[(ii)] $\int_{x}^{+\infty} \overline{F}(t) dt \leq 
\int_{x}^{+\infty} \overline{G}(t)dt$, for all $x \geq F^{-1}(p_{0})$.
\end{itemize}
\end{Corollary}
\begin{proof}
Clearly, (i) and (ii) are equivalent. From $F^{-1}(p_{0}) \leq G^{-1}(p_{0})$ and (ii) we have
\begin{equation*}
\int_{x}^{+\infty} \overline{F}(t) dt \leq \int_{x}^{+\infty} \overline{G}(t)dt, \forall x \geq G^{-1}(p_{0}).
\end{equation*}

Using Theorem \ref{thm_tvar<-&->tcx}(ii) it follows that
$X\leq _{p_{0}\textup{-tvar}}Y$. Conversely, if $X\leq _{p_{0}\textup{-tvar}}Y$, (ii) follows from Theorem \ref{thm_tvar<-&->tcx}(i).
\end{proof}

The following result shows that two random variables, in which distribution functions cross a finite number of times, are ordered in the $p_0$-tail value at risk order for some $p_0 \in (0,1)$. The proof is straightforward and it is omitted.

\begin{Theorem}\label{condsuf}
Let $X$ and $Y$ be two random variables with distribution function $F$ and $G$, respectively. If~\mbox{$S^-(G^{-1}-F^{-1})$} is finite, nonzero and
the last sign change occurs from $-$ to $+$, then $X\leq_{p_n\textup{-tvar}}Y$, where $p_n$ denotes the last crossing point. 
\end{Theorem} 

A random variable $X$ is said to be smaller than $Y$ in the univariate dispersive ordering (denoted~by $X\leq_{disp} Y$) if
$F^{-1}(q)-F^{-1}(p)\leq G^{-1}(q)-G^{-1}(p)$ for all $0< p\leq q< 1$. It is well-known (see~Theorem~2.6.7 in Belzunce et al. [7]) that if $X \leq_{disp} Y$ and $E[X]>E[Y]$ (which, in particular, implies that $X\nleq_{icx }Y$ ), then  $S^-(G^{-1}-F^{-1})=1$. This observation, together with Theorem \ref{condsuf}, allows to find many parametric distributions such that $X\leq_{p_0 \textup{-tvar}}Y$ for some $p_0>0$ and $X\nleq_{icx }Y$. Next, we provide some examples.

\begin{Example} \label{pfamily}
The following examples follow from Theorem \ref{condsuf} using results that are well-known for the dispersive order. Note that $X\nleq_{icx }Y$.
 \begin{enumerate}
\item[(i)] Let $X\sim N(\mu_{1}, \sigma_1)$ and $Y\sim N(\mu_{2},\sigma_2)$ be two normal random variables such that $\mu_{1} > \mu_{2}$ and $\sigma_1 < \sigma_2$. Then, $X \leq_{p_0 \textup{-tvar}} Y$, where $p_0= F_Z(\frac{\mu_1-\mu_2}{\sigma_2-\sigma_1})$ and $Z\sim N(0,1)$.
	\item[(ii)] Let $X\sim \text{Logistic}(\mu_1,\sigma_1)$ and $Y\sim \text{Logistic}(\mu_2,\sigma_2)$ be two logistic random variables such that $\mu_{1} > \mu_{2}$ and $\sigma_1 < \sigma_2$. Then, $X \leq_{p_0 \textup{-tvar}} Y$, where $p_0= F_Z (\frac{\mu_1-\mu_2}{\sigma_2-\sigma_1})$ and $Z\sim \text{Logistic}(0,1)$.
	\item[(iii)] Let $X\sim W(\lambda_1, k_{1})$ and $Y\sim W(\lambda_{2} , k_{2})$ be two Weibull random variables such that $\textup{E}[X] > \textup{E}[Y]$ and $k_{2} < k_{1}$. Then, $X \leq_{p_0 \textup{-tvar}} Y$, where $p_0= F_Z(a_0^{b_0})$, $a_0= \lambda_1 / \lambda_2 $,  $b_0=k_1k_2/ (k_1-k_2)$ and $Z\sim W(1,1)$.	
\item[(iv)] Let $X\sim P(a_1, k_{1})$ and $P\sim P(a_{2} , k_{2})$ be two Pareto random variables such that $\textup{E}[X] > \textup{E}[Y]$ and $k_{2} < k_{1}$. Then, $X \leq_{p_0 \textup{-tvar}} Y$, where $p_0= F_Z(a_0^{b_0})$, $a_0= a_1 / a_2 $,  $b_0=k_1k_2/ (k_1-k_2)$ and $Z\sim P(1,1)$.

\end{enumerate}
\end{Example}

\begin{Remark}
From Proposition \ref{closuretrans} it is apparent that (i) and (ii) in Example \ref{pfamily} are also valid for the log-normal and log-logistic distribution families, respectively.   
\end{Remark}
\begin{Remark}
The tail value at risk is a special case of distortion risk measure. Recall that a  distortion function is a  continuous, and nondecreasing function $h : [0, 1] \rightarrow [0, 1]$ such that $h(0) = 0$ and $h(1) = 1$. Given a random variable $X$ with finite mean, the distorted random variable $X_h$ induced by $h$ has a tail function given by
\begin{equation*}
\overline{F}_{h}(x)=h\left(\overline F(x)\right),
\end{equation*}
for all $x$ in the support of $X$. It is known that given two random variables $X$ and $Y$ such that $X\leq _{icx}Y$ and a concave distortion $h$, entonces $X_{h}\leq _{icx}Y_{h}$ (see Theorem 13 in Sordo et al. [19]). It is straightforward to show (although the proof is notationally cumbersome, so we omit it) that it  $X \leq _{p_{0}\textup{-tvar}} Y$ for $p_{0} \in (0,1)$ and $h$ is a concave distortion function, then $X_{h} \leq_{q_{0}\textup{-tvar}}Y_{h}$ where $q_{0}=G_{Y_h}\left(F^{-1} (p_{0})\right)$. 
\end{Remark}

Before ending this section, we emphasize that, given $p_0 \in (0,1)$,  the order $\leq_{p_0-tvar}$ is a pure tail order in the sense of Rojo [12]. This is shown in the following result, which shows that if  $X\leq_{p_0-tvar} Y$, then the density function $f(x)$ of $X$ decreases faster than the density function $g(x)$ of $Y$.  

\begin{Proposition}
\label{rojo} Let $X$ and $Y$ be two random variables with distribution functions $F$ and $G$ and density functions $f$ and $g$, respectively. Let $p_0 \in (0,1)$. Then
\begin{equation*}
X\leq_{p_0-tvar} Y \text{ implies } \lim_{x\rightarrow +\infty }\frac{f(x)}{g(x)}\leq 1.
\end{equation*}
\end{Proposition}
\begin{proof}
Just applying the L'Hôpital's rule and using Proposition 
\ref{thm_tvar<-&->tcx}(i), we see that
\begin{equation*}
\lim_{x \rightarrow +\infty } \frac{f(x)}{g(x)} = \lim_{x \rightarrow +\infty }\frac{\int_{x}^{+\infty} \overline{F}(t)} {\int_{x}^{+\infty} \overline{G}(t)} \leq 1.
\end{equation*}
\end{proof}

\section{A Real Data Example}\label{sec:5}

In this section, we provide a financial application with a real dataset, involving two random variables of $-$log returns (recall that if $p_{t}$ denotes the price of an asset at day $t$, the corresponding $-$log return is defined by $r_{t} = -log(p_{t} / p_{t-1})$). %We we will obtain empirical evidence to conclude that one of them does not dominate the other with respect to the $icx$ order, but dominates the other one with respect to the $p_{0}\text{-tvar}$ order for same $p_{0}\in(0,1)$. 
Data are of public access and can be obtained from the Yahoo! Finance site. In order to eliminate the time dependent effect, data are related to the weekly close of trading. 

We consider two national stock market indexes: the Mexican $S \& P / BMV$ IPC, denoted by MXX, that measures the performance of the largest and most liquid stocks listed on the Mexican Stock Exchange, and the Hang Seng Index, denoted by HSI, which is the 
main indicator of the performance of the $50$ largest companies of the Hong Kong Stock Exchange. For each index, we have obtained samples of size $n=104$ corresponding to the weekly closings from 1st February 2016 until 31st  January 2018. Let us denote by $R^{MXX}$ and $R^{HSI}$ the $-$log returns of MXX and HSI, respectively. We will obtain empirical evidence to conclude that $R^{MXX} \nleq_{icx}R^{HSI}$ but there exists $p_0\in(0,1)$ such that $R^{MXX} \leq_{p_0\textup{-tvar}}R^{HSI}$.

Before testing the orderings  we test the randomness by a classic runs test to $R^{MXX}$ and $R^{HSI}$ and we obtain the p-values $0.237$ and $0.1149$, respectively. At this point, we plot in Figure \ref{ecdf_MXX_HSI} the empirical probability distributions of $R^{MXX}$ and $R^{HSI}$.   

From Theorem \ref{condsuf} it is apparent that Figure \ref{ecdf_MXX_HSI} shows reasonable empirical evidence that $R^{MXX} \leq_{p_0\textup{-tvar}}R^{HSI}$ for a certain value $p_0\in(0,1)$. In order to see that $R^{MXX} \nleq_{icx}R^{HSI}$ we just need to check the expectations. For such a purpose, we first test the symmetry by performing the $M$, $CM$ and $MMG$ tests described in the lawstat R package. The minimum p-value of the three tests is greater than $0.05$, therefore, we cannot reject the hypothesis that both distributions, $R^{MXX}$ and $R^{HSI}$, are symmetric. Next we compare the medians by  a classical Wilcoxon signed-rank test for paired samples. We obtain that the median of $R^{MXX}$ is greater than the corresponding of $R^{HSI}$, p-value$=0.01$. Therefore, since data are assumed to be symmetric, we conclude that $E[R^{MXX}] > E[R^{HSI}]$, which~implies that $R^{MXX} \nleq_{icx}R^{HSI}$. 

\begin{figure}[H]
	\centering
	\includegraphics[width=0.5\linewidth]{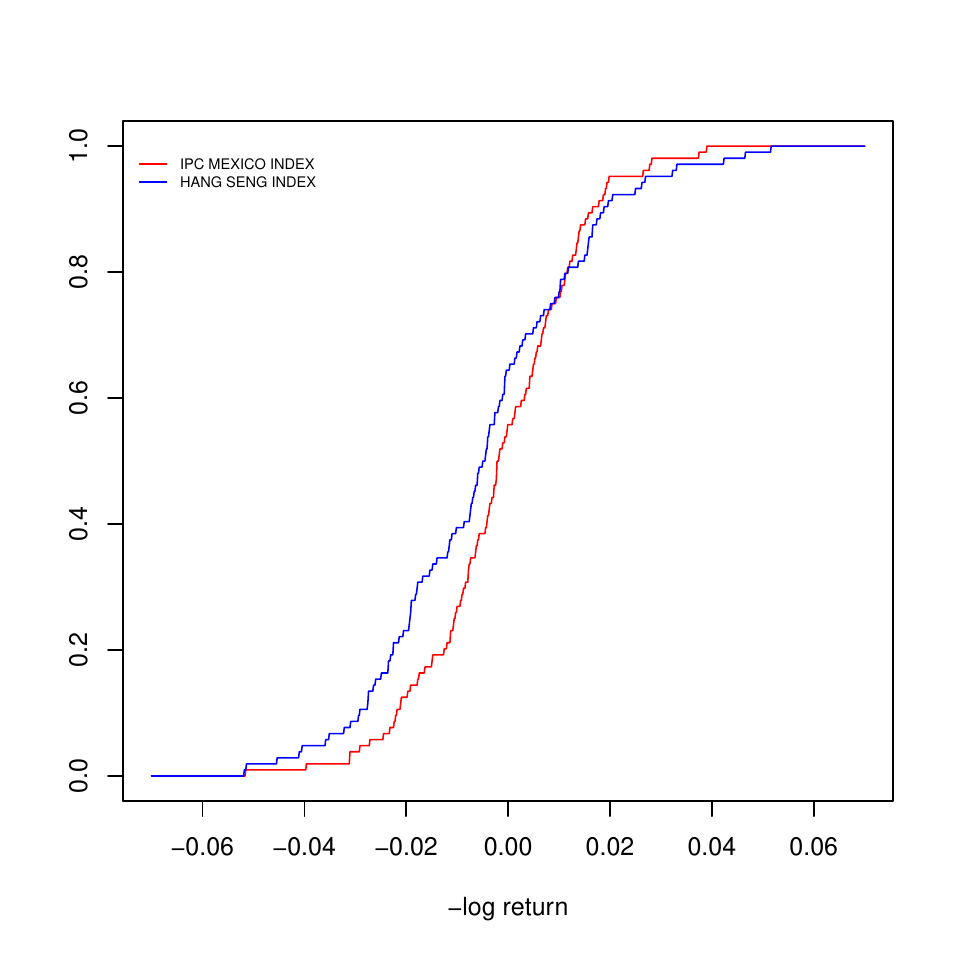}
	\caption{The empirical distribution functions of $R^{MXX}$ and $R^{HSI}$.}
	\label{ecdf_MXX_HSI} 
\end{figure}

The previous non-parametric study suggests to fit a distribution function to $R^{HSI}$ greater than $R^{MXX}$ in terms of $\leq_{p_0\textup{-tvar}}$ order (with bigger tail risk). Two classical distributions for log returns are the normal and logistic distributions, where the latter better reflects an excess of kurtosis. Table \ref{table1} summarizes the \emph{p}-values of the classical K-S goodness of fit test for both distribution families and Figure \ref{FigHist_MXX_HSI} provides the histograms of the log returns together with normal and logistic density estimates
using the maximum likelihood estimates (MLE) for the location and scale parameters given in Table \ref{table2}.

\begin{table}[H] 
	\centering
	\caption{The p-values for fitting normal and logistic distributions.} \label{table1} 
		\begin{tabular}{ccc} 
			\hline
			& \textbf{K-S Goodness of Fit Test (\emph{p}-Value)} &  \\ \hline
			\textbf{Index} &  \textbf{Normal} & \textbf{Logistic}   \\ \hline
			MXX &  $0.0512$ & $0.0347$ \\ \hline
			HSI &  $0.0525$ & $0.0557$ \\ \hline
		\end{tabular} \label{goftest}
%	\end{center}
	
\end{table}
\unskip
\begin{table}[H] 
	\centering
		\caption{The maximum likelihood estimates (MLE) for normal and logistic distributions.}\label{table2} 
		\begin{tabular}{ccccc} 
		\hline
			& \multicolumn{2}{c}{\textbf{Normal}} & \multicolumn{2}{c}{\textbf{Logistic}} \\ \hline
			\textbf{Index} & \boldmath{$\mu$} &\boldmath{$\sigma$} & \boldmath{$\mu$} & \boldmath{$\sigma$} \\ 
			\hline
			MXX & $-0.001486327$ & $0.01556418$ & $-0.001226663$ & $0.008706823$ \\ \hline
			HSI & $-0.005058756$ & $0.02017290$ & $-0.005402021$ & $0.011434130$ \\ 
			\hline
		\end{tabular} 
	%\end{center}

\end{table}

Although Table \ref{table1} suggests than normal distributions fit better than logistics, both can be appropriate. From Example \ref{pfamily} there is enough evidence to assume that $R^{MXX} \leq_{p_0\textup{-tvar}}R^{HSI}$ and the crossing point can be computed from a parametric point of view. In conclusion, a decision maker concerned by the tail value at risk for large values of $p$ will evaluate $R^{MXX}$ as less dangerous than $R^{HSI}$ despite $\textup{E}[R^{MXX}]>\textup{E}[R^{HSI}]$ and $R^{MXX} \nleq_{icx}R^{HSI}$.
\begin{figure}[H]
	\centering
	\begin{subfigure}[b]{0.49\linewidth} 
		\includegraphics[width=\linewidth]{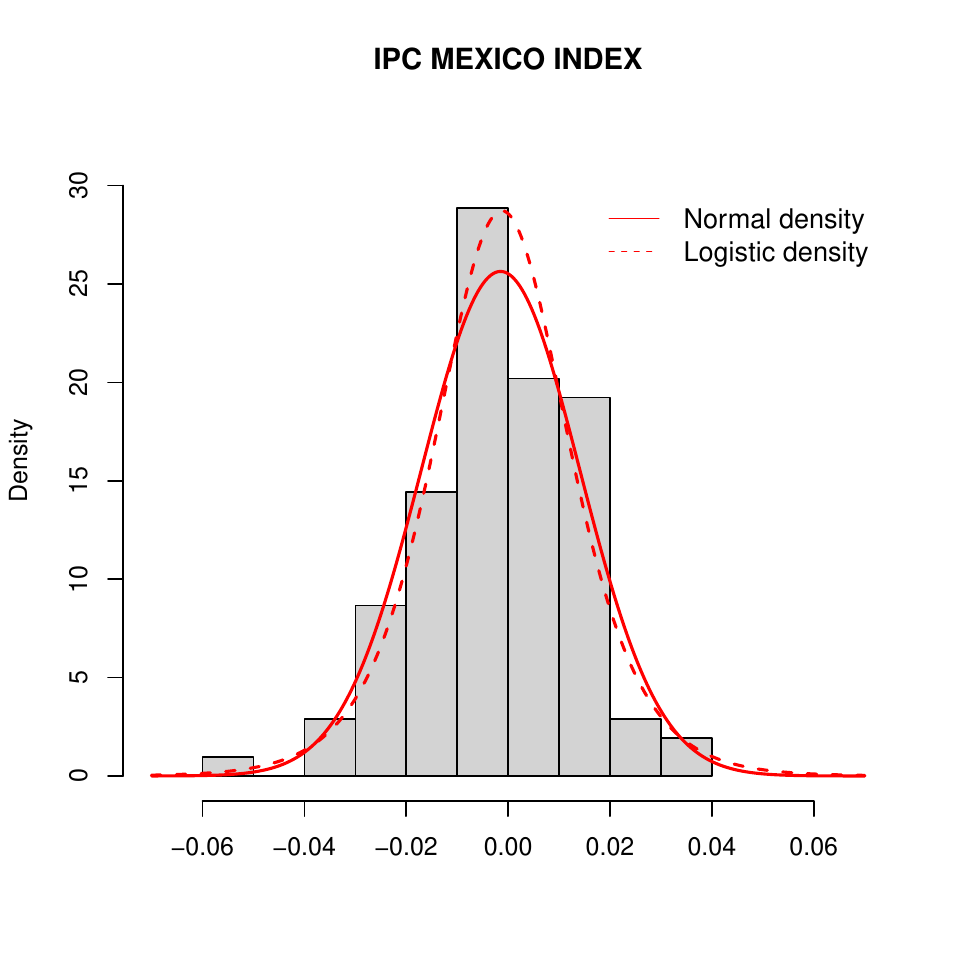}
	\end{subfigure}
	\begin{subfigure}[b]{0.49\linewidth} 
		\includegraphics[width=\linewidth]{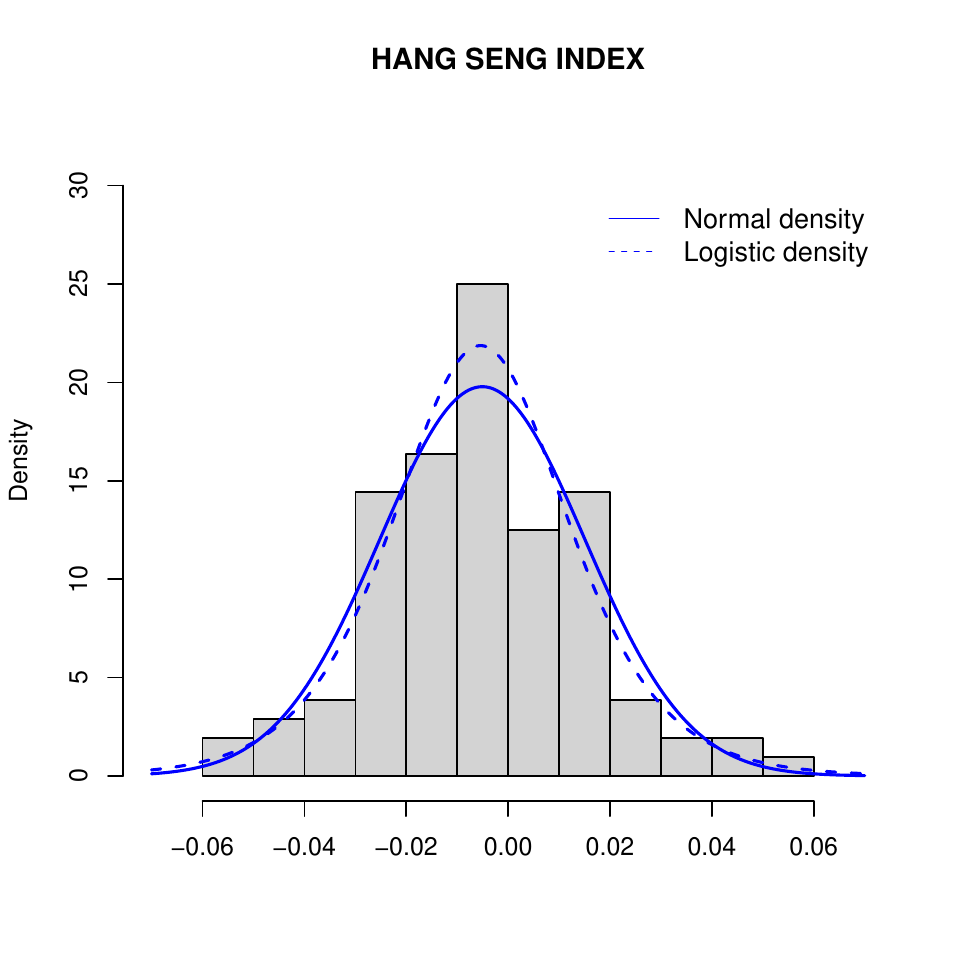}
	\end{subfigure}
	\caption{Histograms of the log
		returns with a normal density
		estimate (solid line) and a logistic density estimate
		(dashed line) superimposed.  The left-hand panel corresponds to $R^{MMX}$ and right-hand panel corresponds to $R^{HSI}$.}
	\label{FigHist_MXX_HSI} 
\end{figure}

\section{Conclusions}\label{sec:6}
 In this paper, we have introduced a family of stochastic orders indexed by confidence levels $p_0 \in (0,1)$, which are useful when we are concerned with right-tail risks. Once $p_0$ is fixed to suit our preferences, we say  that $X$ is less risky than $Y$ if the tail value at risk of $X$ is smaller than the tail value at risk of $Y$ for any confidence levels $p$ such that $p > p_0.$  We have studied the properties of this family of orders as well as its relationships with other stochastic orders, in particular with the tail convex order introduced by  Cheung and Vandulfel [8]. We have illustrated the results with a real financial dataset involving log returns.

\section{Acknowledgments}

The research of Julio Mulero was partially supported by the Conselleria d'Educaci\'o, Investigaci\'o, Cultura i Esport (Generalitat de la Comunitat Valenciana) under grant \textit{GV/2017/015} and the Ministry of Science and Innovation (Spain) under grant \textcolor{red}{\textit{PID2019-103971GB-I00}} (AEI/ FEDER, UE). The research of Bello, Sordo and Su\'arez-Llorens was partially  supported by Ministry of Economy and Competitiveness  (Spain) under grant MTM2017-89577-P, by the 2014-2020 ERDF Operational Programme and by the Department of Economy, Knowledge, Business and University of the Regional Government of Andalusia. Project reference: FEDER-UCA18-107519.

\end{document}